\documentclass[12 pt,a4paper,twoside,reqno,titlepage]{article} 
\usepackage{inputenc,amsmath,amsfonts,amssymb,amsthm,enumitem,lipsum,mathrsfs,hyperref}
\usepackage[none]{hyphenat}
\usepackage{xcolor}
\usepackage[noblocks]{authblk}
\hypersetup{
    colorlinks,
    linkcolor={blue},
    citecolor={blue},
    urlcolor={blue}
}

\title{How to approximate irrational numbers nicely?}
\author{Tirthankar Bhattacharyya\thanks{Department of Mathematics, Indian Institute of Science, Bengaluru, Karnataka, India}\and Soham Bakshi\thanks{Undergraduate Programme, Indian Statistical Institute, Bengaluru, Karnataka, India}\and Arka Das\thanks{Undergraduate Programme, Indian Institute of Science, Bengaluru, Karnataka, India}\and {}}

\date{15 February, 2021}

\theoremstyle{plain}
\newtheorem{thm}{Theorem}[section]
\newtheorem{lem}[thm]{Lemma}
\newtheorem{prop}[thm]{Proposition}
\newtheorem{cor}[thm]{Corollary}
\newtheorem{defn}[thm]{Definition}
\newtheorem{exmp}[thm]{Example}
\newtheorem{note}[thm]{Note}

\begin{document}

\begin{titlepage}
\maketitle
\let\thefootnote\relax
\end{titlepage}

\newpage

\begin{abstract}

In the literature, we have various ways of proving irrationality of a
real number. In this survey article, we shall emphasize on a particular criterion to  
prove irrationality. This is called nice approximation of a number by a sequence of rational 
numbers. This criterion of irrationality is easy to prove and is of great importance. Using it, irrationality of a large class of numbers is proved.  We shall apply this method to prove irrationality of algebraic, exponential and trigonometric irrational numbers. 
Most of the times, explicit constructions are done of the nice sequence of rational numbers. Arguments used are basic like the binomial theorem and maxima and minima of functions. Thus, this survey article is accessible to  undergraduate students and has a broad appeal for anyone looking for entertaining mathematics. This article brings out classically known beautiful mathematics without using any heavyweight machinery. 

\end{abstract}

\tableofcontents

\newpage

\section{Rational approximations}

\subsection{Introduction}
 By definition a {\em rational number} is a real number which can be written in the form $\frac{a}{b}$ where $a$ and $b$ are  integers and $b$ is non-zero. To simplify matters further, we can always take $gcd(a,b)=1$ by just canceling the common factors. A pair of integers $(a,b)$ such that $gcd(a,b)=1$ is called a {\em co-prime pair}. Thus, a rational number is one which can be written as $\frac{a}{b}$ where $a$ and $b$ are co-prime integers. An {\em irrational number} is a real number which is not rational. Though it may seem trivial to determine if a number is rational or not, it is not always the case. We still do not know rationality (or otherwise) of quite a few well known numbers like $\pi + e$, $2^e$, $\pi e$ and $\pi ^{\sqrt{2}}$ for example. Numbers like $\pi$ and $e$ are known to be irrational, although the proofs of their irrationality are not very trivial. In this expository note, we shall concentrate on one particular characterization of irrational numbers.
 
The symbols $\mathbb N$, $\mathbb R$ and $\mathbb Q$ will respectively denote the set of natural numbers, the set of real numbers and the set of rational numbers. To start with, let us note that any real number can be approximated by rational numbers. This is known as the denseness property of the set of rational numbers. We shall not prove this basic result here because it is a very well known consequence of the Archimedean property of real numbers and can be found in many text books. We shall specify a particular way to approximate a real number which will turn out to be a characterisation of $\mathbb R \setminus \mathbb{Q}$. Let us begin with  a general rational approximation (sometimes called the diophantine approximation).

\subsection{Approximation Theorem}
\begin{defn}
Let $\alpha \in \mathbb{R}$. A sequence $\left\{\frac{p_{n}}{q_{n}}\right\}_{n \geq 1}$ of rational numbers is a \emph{rational approximation} of $\alpha$ if
$$0 < \left|\alpha - \frac{p_{n}}{q_{n}}\right| \to 0 \ \textrm{as} \ n \to \infty. $$
It is also written as $\frac{p_{n}}{q_{n}} \to \alpha$.
\end{defn}

Observe that if $\alpha$ is rational, then $\alpha - \frac{p_{n}}{q_{n}}$ can be chosen to be 0 for all $n$. But the sequence $\{\frac{p_{n}}{q_{n}}\}$ can also be chosen in such a way that $\frac{p_{n}}{q_{n}} \to \alpha$ and $\alpha - \frac{p_{n}}{q_{n}}$ is never $0$. This motivates us to define a \emph{nice rational approximation}. 

\begin{defn}
Let $\alpha \in \mathbb{R}$. A sequence $\left\{\frac{p_{n}}{q_{n}}\right\}_{n \geq 1} $ of rational numbers is a \emph{nice rational approximation} of $\alpha$ if
$$0 < \left|q_{n}\alpha - p_{n}\right| \to 0 \ \textrm{as} \ n \to \infty. $$
It is denoted by $\frac{p_{n}}{q_{n}} \xrightarrow{*} \alpha$.
\end{defn}

The terminology {\em nice rational approximation} is according to us. In literature, the same concept has sometimes been called a {\em very good rational approximation}.
The following theorem comes as a bit of a surprise.

\begin{thm}[Approximation Theorem] Let $\alpha \in \mathbb{R}$. There exists a {\em nice} rational approximation of $\alpha$ if and only if $\alpha$ is irrational.
\end{thm}
\begin{proof} Let $\left\{\frac{p_n}{q_n}\right\}_{n=1}^{\infty}$ be a nice rational approximation of $\alpha$. We want to show that $\alpha$ is irrational. If $\alpha$ were rational, then $\alpha = \frac{a}{b}$ for some integers $a$ and $b$ with $b \neq 0$.  Now,
$$\left|q_{n}\alpha - p_{n}\right| = \left|q_{n}\frac{a}{b} - p_{n}\right| = \frac{\left|aq_{n} - bp_{n}\right|}{|b|}.$$
Since $\left|q_{n}\alpha - p_{n}\right| \neq 0$, we have $ \left|aq_{n} - bp_{n}\right| \neq 0$. Thus, $\left|aq_{n} - bp_{n}\right|$ is a positive integer. But then $$\left|q_{n}\alpha - p_{n}\right| = \frac{\left|aq_{n} - bp_{n}\right|}{|b|} \geq \frac{1}{|b|}. $$ 
This is a contradiction as $\left|q_{n}\alpha - p_{n}\right| \to 0$. So, $\alpha$ has to be irrational.

Conversely, if $\alpha$ is given to be irrational, then we want to construct a nice rational approximation of it. We shall construct a sequence  $\left\{\frac{p_n}{q_n}\right\}_{n=1}^{\infty}$ of rational numbers such that $q_{n}\alpha - p_{n} \to 0$. We do not worry about $q_{n}\alpha - p_{n} \neq 0$ because we are already working with an irrational number $\alpha$. Thus, it should be sufficient to find integers $p_{n}$ and $q_{n} \neq 0$ such that $\left|q_{n}\alpha - p_{n}\right| < \frac{1}{n}$, for all$\ n \in \mathbb{N}$. This is done by an ingenious application of the {\em Pigeonhole Principle}.
\par
Let, $n \in \mathbb{N}$ be given. Consider the following $n+1$ real numbers in the half open interval $[0,1)$:
$$0, \ \alpha - \lfloor\alpha\rfloor, \ldots , \ n\alpha -\lfloor{n\alpha}\rfloor $$
where $\lfloor\alpha\rfloor$ denotes the largest integer smaller than or equal to $\alpha$, also called the {\em floor function}. 
Along with the $n+1$ real numbers mentioned above, consider also the  $n$ half open intervals:
$$\left[0,\frac{1}{n}\right), \ \left[\frac{1}{n},\frac{2}{n}\right), \ldots, \ \left[\frac{n-1}{n},1\right).$$
These intervals are disjoint and their union is $[0,1)$, i.e., 
$$\bigcup\limits_{j=0}^{n-1} {\left[\frac{j}{n}, \frac{j+1}{n}\right)} = [0,1).$$
Thus, the $n+1$ real numbers mentioned above are contained in the union of $n$ disjoint intervals. By the {\em Pigeonhole Principle}, there are two numbers among the $n+1$ numbers above which are contained in one of the intervals. 
These two numbers are of the form $n_{1}\alpha - \lfloor n_{1}\alpha \rfloor$ and $ n_{2}\alpha - \lfloor n_{2}\alpha \rfloor$ for some $n_1, n_2 \in \{0,1,\ldots, n\}$. So, there is a  $j \in \{0,1,\ldots, n-1\}$ such that
$$ n_{1}\alpha - \lfloor n_{1}\alpha \rfloor , n_{2}\alpha - \lfloor n_{2}\alpha \rfloor\in \left[\frac{j}{n}, \frac{j+1}{n}\right) . $$
This means that 
$$\left|(n_{1}\alpha - \lfloor n_{1}\alpha \rfloor) - (n_{2}\alpha - \lfloor n_{2}\alpha \rfloor)\right| < \frac{1}{n}.$$
In other words,
$$ |(n_{1} - n_{2})\alpha - (\lfloor n_{1}\alpha \rfloor - \lfloor n_{2}\alpha \rfloor)| < \frac{1}{n}.$$
Now take $q_{n} = n_{1} - n_{2}$ and $p_{n} =  \lfloor n_{1}\alpha \rfloor - \lfloor n_{2}\alpha \rfloor$. Clearly, $q_{n} \neq 0$ and $p_{n}$ are integers satisfying $|q_{n}\alpha - p_{n}| < \frac{1}{n}$. We are done.
\end{proof}

In spite of its simple proof, this theorem remains one of the deepest results in elementary number theory. We shall see its usefulness in this article. In view of the Approximation Theorem, to prove a number irrational we just need to prove that it can be nicely approximated by rational numbers. In many cases, we can explicitly find out the nice rational approximation but there can also be circumstances where giving an existential argument is easier than finding an explicit sequence. We now state an equivalent formulation of the Approximation Theorem. Before that, we need a notation. Recall that the floor function $\lfloor x \rfloor$ is defined to be the greatest integer less than or equal to the real number $x$. The {\em fractional part function} $\{ x \}$ is defined to be the difference between $x$ and $\lfloor x \rfloor$.

\begin{defn}
Let $x$ be a real number. Then the fractional part of $x$ is $\{x\}= x -\lfloor x \rfloor$.
\end{defn}

\begin{thm}
Let $\alpha \in \mathbb{R}$. The following are equivalent:
\begin{enumerate}
    \item There is a nice rational approximation of $\alpha$.
    \item There exists a sequence of integers $\{q_{n}\}_{n \geq 1}$ such that 
    $$\{q_{n}\alpha\}\left(\{q_{n}\alpha\}-1\right) \to 0$$ as $n \to \infty$.
    \item $\alpha$ is irrational.
    \end{enumerate}
\end{thm}

\begin{proof}
If $|q_{n}\alpha - p_{n}| < 1$ then $p_{n}$ should be the closest integer to $q_{n}\alpha$. So, in any nice rational approximation, eventually  $p_{n}$ must be the closest integer to $q_{n}\alpha$. Thus the crux of the matter lies in finding the $q_{n}$. The condition $q_{n}\alpha-p_{n} \to 0 $ as $n \to \infty$ is the same as saying that $q_{n}\alpha$ gets arbitrarily close to an integer for large $n$. Thus, $\frac{p_{n}}{q_{n}} \xrightarrow{*} \alpha$ if and only if $\{q_{n}\alpha\} \to 0$ or $\{q_{n}\alpha\} \to 1$ as $n \to \infty$. Hence, the result. 
\end{proof}

For an $\alpha \in \mathbb{R}$, it may be easier to find $k$ many integer sequences $\left\{a_{n,0}\right\}_{n \geq 1}$ , $\left\{a_{n,1}\right\}_{n \geq 1}$, \ldots, $\left\{a_{n,k}\right\}_{n \geq 1}$ for some $k\geq2$ such that 
$$0 \neq \left|a_{n,k}\alpha^{k} + \dots +a_{n,1}\alpha+ a_{n,0}\right| \to 0  \ \textrm{as} \ n \to \infty$$
rather than finding a nice rational approximation directly. 
We shall see in the following theorem that such sequences are also sufficient to conclude irrationality of a number.

\begin{thm}
\label{thm 1.2.3}
Let $\alpha \in \mathbb{R}$. The following are equivalent:
\begin{enumerate}
    \item There is a nice rational approximation for $\alpha$, say $\frac{p_{n}}{q_{n}}\xrightarrow{*} \alpha$.
    \item There exists a natural number $k \geq 1$ for which there are integer sequences $\{a_{n,0}\}_{n \geq 1}$ , $\{a_{n,1}\}_{n \geq 1}$, \ldots, $\{a_{n,k}\}_{n \geq 1}$ such that 
$$0 \neq \left|a_{n,k}\alpha^{k} + \dots +a_{n,1}\alpha+ a_{n,0}\right| \to 0  \ \textrm{as} \ n \to \infty.$$
   \item $\alpha$ is irrational.
\end{enumerate}
\end{thm}

\begin{proof}
The only non-trivial part of the proof is that (2) $\implies$ (3). We shall argue by contradiction. Assume $\alpha \in \mathbb{Q}$. Then $\alpha= \frac{a}{b}$ for some co-prime integer pair $(a,b)$. Then 
$$0 \neq \left|a_{n,k}\left(\frac{a}{b}\right)^{k} + \dots +a_{n,1}\frac{a}{b}+ a_{n,0}\right| \geq \frac{1}{b^k}.$$
Thus 
$$ \left|a_{n,k}\alpha^{k} + \dots +a_{n,1}\alpha+ a_{n,0}\right| \nrightarrow 0 \ \textrm{ as } \ n\to \infty.$$
That is a contradiction. 
\end{proof}

We shall now apply the results obtained in this section to concrete examples.

\section{Algebraic irrationals}

In this section, we shall discuss irrationality of algebraic irrational numbers. We shall start from the simplest algebraic irrational numbers of the form $\sqrt{m}$, where $m\in \mathbb{N}$ and $m$ is not a perfect square. In this case, we shall produce a nice rational approximation explicitly. In the later subsections, we shall use several derived results to conclude about irrationality of a few other complicated irrational numbers indirectly.

\subsection{Finding a nice rational approximation for $\sqrt{m}\notin \mathbb{Z}$}
\begin{thm}
\label{thm 3.1}
Suppose $m\in \mathbb{N}$ is not the square of an integer. Let $z=\left\lfloor \sqrt{m}\right \rfloor$, where $\lfloor.\rfloor$ denotes the greatest integer function. Consider the integers 
$$p_n=\sum_{k=1}^{n}\binom{2n-1}{2k-1}{m}^{\left(n-k\right)}{z}^{2k-1}  \text{ and } q_n=\sum_{k=0}^{n-1}\binom{2n-1}{2k}{m}^{n-1-k}{z}^{2k}.$$
Then $\left\{\frac{p_n}{q_n}\right\}_{n \ge 1}$ is a nice rational approximation to $\sqrt{m}.$
\end{thm}
\begin{proof}
As $m\ge 2$, we have $z\in\mathbb{N}$.  Also, $0<\left(\sqrt{m}-z\right)<1$. A binomial series expansion gives that 
\begin{align*}
& \left(\sqrt{m}-z\right)^{2n-1} \\
= & \sum_{j=0}^{2n-1}\binom{2n-1}{j}(\sqrt{m})^{2n-1-j}(-z)^{j}\\
= & \sum_{k=0}^{n-1}\binom{2n-1}{2k}(\sqrt{m})^{2n-1-2k}(-z)^{2k}+\sum_{k=1}^{n}\binom{2n-1}{2k-1}(\sqrt{m})^{\left(2n-1-(2k-1)\right)}(-z)^{2k-1}\\
= & \sqrt{m} \sum_{k=0}^{n-1}\binom{2n-1}{2k}{m}^{n-1-k}{z}^{2k}-\sum_{k=1}^{n}\binom{2n-1}{2k-1}{m}^{\left(n-k\right)}{z}^{2k-1}=\sqrt{m} q_n - p_n.\end{align*}
Now,  $0 \neq \left(\sqrt{m}-z\right)^{2n-1} \to 0+$ because $0<\left(\sqrt{m}-z\right)<1$. So, $$0 \neq  \left|\sqrt{m} q_n - p_n\right| \to 0. $$
\end{proof}

This implies the irrationality of  $\sqrt{m}$ when $\sqrt{m} \notin \mathbb{Z}$. 

\subsection{Irrationality of $\sqrt[m]{a} \notin \mathbb{N}$} 
In this subsection, we shall prove irrationality for $\sqrt[m]{a}$ where $a$ and $m$ are natural numbers but $\sqrt[m]{a}$  is not. 

\begin{thm}
\label{thm 3.2}
If $a,m \in \mathbb{N}$ and $\sqrt[m]{a} \notin \mathbb{N}$, then $\sqrt[m]{a}$ is irrational.
\end{thm}
\begin{proof}
Take $z=\lfloor \sqrt[m]{a}\rfloor$. By the assumption $\sqrt[m]{a} \notin \mathbb{N}$, we have $z \neq \sqrt[m]{a} $. So, $0<\left(\sqrt[m]{a}-z\right) < 1$. If $\alpha=\sqrt[m]{a}$, then
\begin{align*}\left(\sqrt[m]{a}-z\right)^{mn-1}  & =\sum_{j=0}^{mn-1}\binom{mn-1}{j}\left(\sqrt[m]{a}\right)^{j}(-z)^{mn-1-j} \\
& =\sum_{l=0}^{m-1}\sum_{k=0}^{n-1} \binom{mn-1}{mk+l} \left(\sqrt[m]{a}\right)^{mk+l}(-z)^{mn-1-(mk+l)}\\
& =\sum_{l=0}^{m-1} \alpha^{l} \sum_{k=0}^{n-1} \binom{mn-1}{mk+l} a^{k}(-z)^{mn-1-(mk+l)}\\
& =\sum_{l=0}^{m-1} a_{n,l} {\alpha^{l}},\end{align*}
where $$ a_{n,l} =\sum_{k=0}^{n-1} \binom{mn-1}{mk+l} a^{k}(-z)^{mn-1-(mk+l)} \in \mathbb{Z}.$$
As $n \to \infty $, $0 \neq \left(\sqrt[m]{a}-z\right)^{mn-1} \to 0+$, so, $$0 \neq \left| \sum_{l=0}^{m-1} a_{n,l} {\alpha^{l}} \right| \to 0$$ 
So, from Theorem \ref{thm 1.2.3}, $\alpha=\sqrt[m]{a}$ is irrational.
\end{proof}

\subsection{General algebraic irrational numbers}

We shall start with a lemma, which we shall need to prove our results. It is also of an independent interest.

\begin{lem}
\label{lem 3.3}
Let $\alpha$ be a root of a monic polynomial $x^m+\sum_{k=0}^{m-1} b_k x^k$ with integer coefficients. If $n$ is any natural number and $c_0$, $c_1$,  \ldots , $c_n$ are integers, then there exist $m$ integers $d_0$, $d_1$, \ldots , $d_{m-1}$ such that 
$$\sum_{k=0}^{n} c_k \alpha^k =\sum_{k=0}^{m-1} d_k \alpha^k .$$
\end{lem}

\begin{proof}
We shall discuss an algorithm, by which,  we can get the expression $\sum_{k=0}^{n} c_k \alpha^k =\sum_{k=0}^{m-1} d_k \alpha^k $. The result will be automatically proved when we justify the algorithm. 

If $n<m$, then we are done, as $d_k=c_k$, for $0\le k \le n$ and $d_k=0$, for $n < k \le m$ shall serve the purpose.

For $n \ge m$, note that $\alpha^m=- \sum_{k=0}^{m-1} b_k \alpha^k$, which means that 
\begin{align*}\sum_{k=0}^{n} c_k \alpha^k & =  c_n \alpha^{(n-m)} \cdot \alpha^m+\sum_{k=0}^{n-1} c_k \alpha^k \\
& =c_n \alpha^{(n-m)}\left(- \sum_{k=0}^{m-1} b_k \alpha^k\right)+\sum_{k=0}^{n-1} c_k \alpha^k \\
&=\sum_{k=n-m}^{n-1}{\left(c_k-b_{k-(n-m)} c_n\right) \alpha^k}+\sum_{k=0}^{n-m-1} c_k \alpha^k \\
& \text{ (the second sum appears only when } n-m-1 \ge 0 \text{ )} \\
&=\sum_{k=0}^{n-1} c'_k \alpha^k \end{align*}
for some integers $c'_k$ for $k = 0,1, \ldots ,n-1$.
Now, if $n=m$, we are done. If not, continue this process, until an expression is obtained, in which the highest power of $\alpha$ is less than $m$. This will be achieved in finitely many steps because the highest power of $\alpha$ decreases by $1$ in each step. So,
$$\sum_{k=0}^{n} c_k \alpha^k =\sum_{k=0}^{n-1} c'_k \alpha^k = \cdots = \sum_{k=0}^{m-1} d_k \alpha^k$$ for some integers $d_0, d_1, \ldots , d_{m-1}.$

\end{proof}

\begin{thm}
\label{thm 3.4}
If $\alpha \notin \mathbb{Z}$ is a root of a monic polynomial with integer coefficients, then $\alpha$ is irrational.
\end{thm}

\begin{proof}
We again use the binomial theorem to our advantage. Let $\alpha$ be the root of the monic polynomial $x^m+\sum_{k=0}^{m-1} b_k x^k$. As $\alpha$ is not an integer, $\alpha \neq \lfloor \alpha \rfloor =z $ (say).
Now,
$$
0 \neq (\alpha - z)^n = \sum_{k=0}^{n} \binom{n}{k} \alpha^{k}(-z)^{n-k} = \sum_{k=0}^{n} c_{n,k} \alpha^k$$
where $c_{n,k} = \binom{n}{k}(-z)^{n-k} \in \mathbb{Z}$. By Lemma \ref{lem 3.3}, we have the last quantity to be 
$\sum_{k=0}^{m-1} d_{n,k} \alpha^k$
for some integers $d_{n,k}$. As $0<\alpha-z<1$, we have $0 \neq (\alpha-z)^n \to 0+$ as $n \to \infty $. So,
$$ 0 \neq \left| \sum_{k=0}^{m-1} d_{n,k} \alpha^k \right| \to 0.$$
So, from Theorem \ref{thm 1.2.3}, $\alpha$ is irrational.
\end{proof}

An obvious extension of the previous theorem is presented below. Before that, we need to prove a little lemma.

\begin{lem}
\label{lem 3.5}
Let $a \in \mathbb{Z}-\{0\} $. If $a\alpha$ is irrational for some $\alpha \in \mathbb{R}$, then $\alpha$ is irrational.

\end{lem}
\begin{proof}
As $a\alpha$ is irrational, there exists some nice rational approximation  $\left\{ \frac{p_n}{q_n}\right\}_{n \in \mathbb{N}}$ of $a\alpha$. So,
$$ 
0<\left| q_n a \alpha - p_n\right| \to 0
$$
as $n\to \infty$. It is evident that $\left\{\frac{p_n}{aq_n}\right\}_{n \in \mathbb{N}}$ is a nice rational approximation of $\alpha$. So, $\alpha$ is irrational.

\end{proof}

\begin{thm}
\label{thm 3.6}
If $\alpha$ is a root of a polynomial with integer coefficients and with non-zero leading coefficient $a$, and if $a\alpha \notin \mathbb{Z}$, then $\alpha $ is irrational.

\end{thm}

\begin{proof}
Let $\alpha$ be the root of the polynomial $f(x)=ax^m+\sum_{k=0}^{m-1} c_k x^k$ with $a \in \mathbb{Z}-\{0\}$, where $c_0$, $c_1$,..., $c_{m-1}$ are integers.
.So, the polynomial
$$
a^{m-1} f(x) =a^m x^m + \sum_{k=0}^{m-1} c_k a^{m-1} x^k=(ax)^{m}+ \sum_{k=0}^{m-1} c_k a^{m-k-1} (ax)^k$$
has a root $\alpha$.
So,the polynomial  
$$g(z) = z^m + \sum_{k=0}^{m-1} c_k a^{m-k-1} z^k $$ is a monic polynomial with integer coefficients with a root $a\alpha$.
So, from Theorem \ref{thm 3.4}, as $a\alpha \notin \mathbb{Z} $, $a\alpha$ is irrational. As $a \in \mathbb{Z}-\{0\}$, from Lemma \ref{lem 3.5}, $\alpha$ is irrational.
\end{proof}

\begin{note}
If $\alpha$ is a root of a polynomial with integer coefficients and with non-zero leading coefficient $a$, and if $a\alpha \in \mathbb{Z}$, i.e., if $a\alpha=z$ for some $z \in \mathbb{Z}$, then $\alpha = \frac{z}{a}$ is rational.
\end{note}

Let us try an example to see how the above theorem can be applied.
 
\begin{exmp}{\em 
We want to determine the irrationality of the real roots of the polynomial $f(x)=2x^3-5x^2+x+1$.
Using our knowledge of calculus for determining maxima-minima of a function, we see that $f(x)$ has three real roots.

We note that $f\left(-\frac{1}{2}\right)=-1<0$ and $f(0)=1>0$. Also, $f\left(\frac{1}{2}\right)=\frac{1}{2}>0$ and $f(1)=-1<0$.Finally, $f(2)=-1<0$ and $f\left(\frac{5}{2}\right)=\frac{7}{2}>0$. So, if $\alpha<\beta<\gamma$ are the three real roots of $f(x)$, then from the continuity of $f$ over $\mathbb{R}$, we get that $-\frac{1}{2}<\alpha<0$, $\frac{1}{2}<\beta<1$ and $2<\gamma<\frac{5}{2}$.

So, $ -1<2\alpha<0$, $1<2\beta<2$ and $4<2\gamma<5$.

Here, each of $\alpha$, $\beta$ and $\gamma$ is a root of the polynomial $f(x)$ with integer coefficients and leading coefficient 2. Also, none of $2\alpha$, $2\beta$ and $2\gamma$ is integer, as is clear from the above discussion. So, from Theorem \ref{thm 3.6}, it is evident that each of $\alpha$, $\beta$ and $\gamma$ is irrational.}
\end{exmp}

For further reading, we mention Roth's Theorem which says that \textit{every irrational algebraic number $\alpha$ has approximation exponent equal to 2.} This means that the largest possible value of $\mu$, for which $0 < \left| \alpha - \frac{p}{q} \right| < \frac{1}{q^{\mu}}$, has infinitely many solutions for $p,q \in \mathbb{Z}$ and $q > 0$, is 2. So, for an algebraic irrational number $\alpha$, for $\mu=2$, as there are infinitely many solutions for such $p$, $q$, we can find for each $n \in \mathbb{Z}$,  some integer $q_n >n$ and a corresponding integer $p_n$, such that $0 < \left| \alpha - \frac{p_n}{q_n} \right| < \frac{1}{q_n^{2}}$. So, 
$$0<\left| q_n \alpha - p_n \right| < \frac{1}{q_n} < \frac{1}{n}.$$ 
So, $\left\{ {\frac{p_n}{q_n}}\right\} \xrightarrow{*} \alpha $. See \cite{Roth} for more on Roth's Theorem.

 There is a stronger version of the theorem, namely Hurwitz's theorem, which states the following. \textit{For every irrational number $\alpha$, there exist infinitely many coprime integers $p$ and $q$, such that $\left| \alpha - \frac{p}{q} \right| < \frac{1}{\sqrt{5}q^{2}}$.}

\section{$e$ and other related irrationals}

The Euler's constant $e$ can be defined in several ways, for our purpose, we shall adopt the definition that $e$ is the sum of the convergent series $\sum_{n=0}^{\infty}\frac{1}{n!}= 1 + \frac{1}{1!} +\frac{1}{2!} +\cdots $. The exponential function $e^{x}$ is defined as
$$e^{x} = \sum_{n=0}^{\infty}\frac{x^{n}}{n!}$$
which converges for all real $x$.

\begin{note}
We would like to mention here that this series converges absolutely and uniformly on any compact interval, is differentiable and the derivative is the same as the series again. We shall not prove these elementary properties here because detailed discussions on these topics are available in many standard texts on real analysis.
\end{note}

We shall start by proving irrationality of $e$.

\subsection{Irrationality of  e }
\begin{thm}\label{thm 3.1.1}
Let $p_{n}= \sum_{i=0}^{n} \frac{n!}{i!} = n!+\frac{n!}{1!}+\cdots \frac{n!}{n!} $ and $q_{n}= n!$. Then $\left\{\frac{p_{n}}{q_{n}}\right\}_{n\geq1}$ is a nice rational approximation of $e$. 
\end{thm}

\begin{proof}
 We know that
$$e = \lim_{n \to \infty}{\sum_{k=0}^{n}{\frac{1}{k!}}} = 1 + \frac{1}{1!} + \frac{1}{2!} + \dots$$
With $p_n$ and  $q_{n}$ as in the statement of the theorem, we have $$\frac{p_{n}}{q_{n}} = \sum_{k=0}^{n}{\frac{1}{k!}}.$$
So, 
$$e = 1 + \frac{1}{1!} + \frac{1}{2!} + \cdots = \frac{p_{n}}{q_{n}} + \sum_{k=n+1}^{\infty}{\frac{1}{k!}}$$
This implies that
\begin{align*}
q_{n}e - p_{n} & = q_{n}\sum_{k=n+1}^{\infty}{\frac{1}{k!}} \\
& = \sum_{k=n+1}^{\infty}{ \frac{n!}{k!}}\\
& = \frac{1}{n+1} + \frac{1}{(n+1)(n+2)} + \cdots \\
&  < \frac{1}{n+1} + \frac{1}{(n+1)^2} + \dots = \frac{1}{n}. \end{align*}
Thus $q_{n}e - p_{n} \to 0$ as $n \to \infty$ and also clearly $q_{n}e - p_{n} > \frac{1}{n+1} > 0$, for all $n \in \mathbb{N}$.
Thus $\frac{p_{n}}{q_{n}} \xrightarrow{*} e$ as $n \to \infty$ for the given sequence of integers $\{p_{n}\}_{n \geq 1}$ and $\{q_{n}\}_{n \geq 1}$.
\end{proof}

Nice rational approximations do not behave very nicely under multiplication. Suppose $\frac{p_{n}}{q_{n}} \xrightarrow{*} \alpha$ and $\frac{p'_{n}}{q'_{n}} \xrightarrow{*} \beta$. Then it is not necessary that $\frac{p_{n}p'_{n}}{q_{n}q'_{n}} \xrightarrow{*} \alpha\beta$ simply because product of irrational numbers need not be irrational. Of course, $\left\{\frac{p_{n}p'_{n}}{q_{n}q'_{n}}\right\}_{ n\geq1}$ always converges to $\alpha\beta$ by the multiplicative property of limits. 

The answer to even the simpler question of  whether $\frac{p_{n}^{2}}{q_{n}^{2}} \xrightarrow{*} \alpha^{2}$ when $\frac{p_{n}}{q_{n}} \xrightarrow{*} \alpha$   is in the negative because square of an irrational number need not be irrational. However, is the answer in the affirmative when $\alpha^2$ is irrational? If we look closely at the nicely approximating sequence $\frac{p_{n}}{q_{n}} \xrightarrow{*} \alpha$, we see that  
$$q_{n}^{2}\alpha^{2}-p_{n}^{2} = (q_{n}\alpha+p_{n})(q_{n}\alpha-p_{n})$$
and $q_{n}\alpha + p_{n}$ can diverge as $n \to \infty$ and dominate $q_{n}\alpha-p_{n}$. Here is a concrete example. 

\begin{exmp} \label{exmp 1}
In Theorem \ref{thm 3.1.1}, we saw that if $p_{n} = \sum_{i=0}^{n}\frac{n!}{i!} $ and $q_{n} = n!$ then $ \frac{p_{n}}{q_{n}} \xrightarrow{*} e$. Note that $q_{n} e \geq n!$ and $p_{n} \geq 0$. So, $q_{n}e+p_{n}\geq n!$. Now, 
\begin{align*}q_{n}^{2} e^{2} - p_{n}^{2} & =  (q_{n}e+p_{n})(q_{n}e-p_{n})\\
& \geq  n!\left(\frac{1}{(n+1)}+\frac{1}{(n+1)(n+2)} + \cdots\right)\\
& \geq  \frac{n!}{n+1} \to \infty \ \textrm{as} \ n \to \infty.\\ \end{align*}
Thus $\left\{\frac{p_{n}^{2}}{q_{n}^{2}}\right\}_{n\geq1} $ is not a nice rational approximation for $e^2$. 
\end{exmp}

\subsection{Irrationality of $e^{2}$}
Naturally one can ask about higher powers of $e$. For now, let us try to find nice rational approximations for $e^2$. We have already seen in Example \ref{exmp 1} that a possible approach will not work.

\begin{prop} \label{prop 4.3}
Let $p_{n}= \sum_{i=0}^{n} (-1)^{i}\frac{n!}{i!} = n!-\frac{n!}{1!}+ \cdots (-1)^{n}\frac{n!}{n!} $ and $q_{n}= n!$. Then $\left\{\frac{p_{n}}{q_{n}}\right\}_{n\geq1}$ is a nice rational approximation of $\frac{1}{e}$.
\end{prop}

\begin{proof}
By definition
$$\frac{1}{e}= \sum_{k=0}^{\infty}\frac{(-1)^{k}}{k!}.$$
Now $$\frac{p_{n}}{q_{n}}=\sum_{k=0}^{n}\frac{
(-1)^k}{k!}.$$
Thus
$$\frac{1}{e}=\frac{p_{n}}{q_{n}}+\sum_{k=n+1}^{\infty}\frac{(-1)^{k}}{k!}$$
implying that
$$\left|q_{n}\frac{1}{e}-p_{n}\right|= \left| \sum_{k=n+1}^{\infty}\frac{(-1)^{k}n!}{k!} \right| < \sum_{k=n+1}^{\infty}\frac{n!}{k!} < \frac{1}{n}.$$
Thus $\left|q_{n}\frac{1}{e}-p_{n}\right| \to 0$ as $n \to \infty$. It is easy to check that $\left|q_{n}\frac{1}{e}-p_{n}\right| > 0$. 
\end{proof}

\begin{prop} \label{prop 1.3.1}
Let $\alpha \in \mathbb{R}$. If $\frac{p_{n}}{q_{n}} \xrightarrow{*} \alpha$, then $\frac{q_{n}}{p_{n}} \xrightarrow{*} \frac{1}{\alpha}.$
\end{prop}
\begin{proof}
As $\frac{p_{n}}{q_{n}} \xrightarrow{*} \alpha$, we have $\alpha$ to be irrational. Thus  $\alpha \neq 0$. Now as $|q_{n}\alpha-p_{n}| \to 0$ as $n \to \infty$, we have 
$$0 < \frac{1}{\alpha}\left|q_{n}\alpha-p_{n}\right| = \left |p_{n}\frac{1}{\alpha}-q_{n} \right| \to 0 \ \textrm{as} \ n \to \infty.$$
Thus $\frac{q_{n}}{p_{n}} \xrightarrow{*} \frac{1}{\alpha}$.
\end{proof}

We now have another rational approximation of $e$.

\begin{cor} \label{NewApproxFore}
Let $p_{n}= n!$ and $q_{n}= \sum_{i=0}^{n} (-1)^{i}\frac{n!}{i!} = n!-\frac{n!}{1!}+ \cdots (-1)^{n}\frac{n!}{n!} $. Then $\left\{\frac{p_{n}}{q_{n}}\right\}_{n\geq1}$ is a nice rational approximation of $e$.
\end{cor}
\begin{proof}
By Proposition \ref{prop 1.3.1}, we know that if $\frac{p_{n}}{q_{n}} \xrightarrow{*} \alpha$, then $\frac{q_{n}}{p_{n}} \xrightarrow{*} \frac{1}{\alpha}$. By Proposition \ref{prop 4.3}, we know that
$$\frac{\sum_{i=0}^{\infty}(-1)^{i}\frac{n!}{i!}}{n!} \xrightarrow{*} \frac{1}{e}.$$
Thus combining the two we get
$$\frac{n!}{\sum_{i=0}^{\infty}(-1)^{i}\frac{n!}{i!}} \xrightarrow{*} e.$$ Thus for the given $p_{n}$ and $q_{n}$ 
$\left\{\frac{p_{n}}{q_{n}}\right\}_{n\geq1}$ is a nice rational approximation of $e$.
\end{proof}

We shall use this new nice rational approximation of $e$ to prove irrationality of $e^2$ by virtue of the following proposition.

\begin{prop} \label{alpha beta}
Let $\alpha, \beta \in \mathbb{R}$, $\frac{p_{n}}{q_{n}} \xrightarrow{*} \alpha$ and $\frac{q_{n}}{q'_{n}} \xrightarrow{*} \beta$. Then 
$$q'_{n} \alpha\beta -  p_{n}  \to 0  \text{ as } n \to \infty.$$ 
In other words if $q'_{n} \alpha\beta -  p_{n} \neq 0$, for every $n$ in $\mathbb N$, then $\frac{p_{n}}{q'_{n}} \xrightarrow{*} \alpha\beta$.
\end{prop}

\begin{proof}

We know that $\left|q_{n} \alpha - p_{n}\right| \to 0$ and $\left|q'_{n} \beta - q_{n}\right| \to 0$ as $n \to \infty$. Now,
\begin{align*}\left|q'_{n}\alpha\beta - p_{n}\right| & = \left|q'_{n}\alpha\beta -q_{n} \alpha + q_{n} \alpha - p_{n}\right| \\ 
& = \left|\alpha \left(q'_{n}\beta-q_{n} \right) + \left(q_{n}\alpha - p_{n}\right)\right| \\
& \leq |\alpha||q'_{n}\beta -q_{n}| + |q_{n}\alpha -p_{n}| \to 0 \ \textrm{as} \ n \to \infty.\end{align*}

\end{proof}

 This result can be generalized a bit further. We shall not need it in the sequel, but record it for its own merit. 

\begin{prop}
Let $\alpha, \beta \in \mathbb{R}$, $\frac{p_{n}}{q_{n}} \xrightarrow{*} \alpha$ , $\frac{p'_{n}}{q'_{n}} \xrightarrow{*} \beta$. Also assume $p'_{n}|q_{n}$, for all $n \in \mathbb{N}$ (say $q_{n}=p'_{n}d_{n}$) and there exists $M \in \mathbb{R}$  such that $\left|d_{n}\right| \leq M$, for all $n \in \mathbb{N}$ ($\{d_{n}\}_{n\geq1}$ is a bounded sequence). Then $\left|d_{n}q'_{n} \alpha\beta -  p_{n}\right| \to 0$ as $n \to \infty$. In other words if $\left|d_{n}q'_{n} \alpha\beta -  p_{n}\right| \neq 0$, then $\frac{p_{n}}{d_{n}q'_{n}} \xrightarrow{*} \alpha\beta$.
\end{prop}

\begin{proof}
We have $\left|p'_{n}d_{n}\alpha-p_{n}\right| \to 0$ and $\left|q'_{n}\beta - p'_{n}\right| \to 0$ as $n \to \infty$. Now
$$\left|q'_{n}d_{n}\alpha\beta - p_{n}\right|= \left|q'_{n}d_{n}\alpha\beta - p'_{n}d_{n}\alpha + q_{n}\alpha - p_{n}\right|$$
$$=\left|d_{n}\alpha \left(q'_{n}\beta - p'_{n} \right) + \left(q_{n}\alpha-p_{n}\right) \right| \leq M|\alpha| \left|q'_{n}\beta - p'_{n}\right| + \left|q_{n}\alpha-p_{n}\right| \to 0 \ \textrm{as} \ n \to \infty.$$
\end{proof}

We are ready to find a nice rational approximation for $e^2$.

\begin{thm}
If 
$$p'_{n}= \sum_{i=0}^{2n} \frac{(2n)!}{i!} = (2n)!+\frac{(2n)!}{1!}+\cdots + \frac{(2n)!}{(2n)!} $$ 
and 
$$q'_{n}= \sum_{i=0}^{2n} (-1)^{i}\frac{(2n)!}{i!} = (2n)!-\frac{(2n)!}{1!}+ \cdots + (-1)^{2n}\frac{(2n)!}{(2n)!}, $$ 
then $\left\{\frac{p'_{n}}{q'_{n}}\right\}_{n\geq1}$ is a nice rational approximation of $e^2$.
\end{thm}

\begin{proof}
Define integer sequences 
$$p_{n}= \sum_{i=0}^{n} \frac{n!}{i!} = n!+\frac{n!}{1!}+\cdots + \frac{n!}{n!},$$
$$q_{n}= \sum_{i=0}^{n} (-1)^{i}\frac{n!}{i!} = n!-\frac{n!}{1!}+ \cdots + (-1)^{n}\frac{n!}{n!}$$
and $r_{n}= n!$ for all $n \in \mathbb{N}$. It is clear from Theorem \ref{thm 3.1.1} and Corollary \ref{NewApproxFore}  that $\frac{p_{n}}{r_{n}} \xrightarrow{*} e$ and $\frac{r_{n}}{q_{n}} \xrightarrow{*} e$. Thus applying Proposition \ref{alpha beta}, we get
$$\left|q_{n}e^{2}-p_{n}\right|\to 0 \ \textrm{as} \ n \to \infty.$$ 
It is easy to check that $e>\frac{p_{n}}{n!}$.
Now, from Proposition \ref{prop 4.3}, $\frac{q_{n}}{n!} \xrightarrow{*} \frac{1}{e}$. Moreover, note that for even $n$,
$$\frac{1}{e}<\frac{q_{n}}{n!}.$$
So, if we take $p'_{n}=p_{2n}$ and $q'_{n}=q_{2n}$, then
$$\left|q'_{n}e^{2}-p'_{n}\right|\to 0 \ \textrm{as} \ n \to \infty $$
 and
$$ e^{2}> \frac{p'_{n}}{q'_{n}}, \textrm{ i.e., } \left|q'_{n}e^{2}-p'_{n}\right|>0.$$
Hence, $\left\{\frac{p'_{n}}{q'_{n}}\right\}_{n\geq1}$ is a nice rational approximation of $e^2$.
\end{proof}

\subsection{General higher powers}
What about even higher powers of $e$? It turns out that to continue the above approach to get nice rational approximations of $e^n$ for $n\geq 3$ is difficult. This leads us to \emph{Niven's polynomials}.

\begin{defn}
Let $n\in \mathbb{N}$. Niven's polynomial of degree $2n$ is the function $f_{n}:[0,1] \to \mathbb{R}$ given by
$$f_{n}(x)=\frac{x^{n}(1-x)^{n}}{n!}.$$
\end{defn}

Before proceeding further, let us note a few useful properties of Niven's polynomials. Sometimes, when there is no chance of confusion, we shall drop the suffix $n$ from $f_{n}$ and write it just as $f$.

\begin{prop}
Let $f:[0,1]\to \mathbb{R}$ (in our previous notation $f_{n}$) be the Niven's polynomial with degree $2n$. Then the following statements hold:
\begin{enumerate}
    \item $f(x)=f(1-x)$, for all $x \in [0,1].$
    \item $f(x) \in \left[0,\frac{1}{n!}\right)$, for all $x \in [0,1].$
    \item $f^{(j)}(0),f^{(j)}(1) \in \mathbb{Z}$, for all $j \in \mathbb{N}.$ ($f^{(j)}$ denotes the $j$th derivative of $f$)
\end{enumerate}
\end{prop}

\begin{proof}
1. This proof is straightforward.  $$f(x)=\frac{x^{n}(1-x)^{n}}{n!}=\frac{(1-x)^{n}x^{n}}{n!}=f(1-x).$$
2. Observe that both $x^{n}$ and $(1-x)^{n}$ cannot be $1$ at the same time. Thus, 
$$f(x)=\frac{x^{n}(1-x)^{n}}{n!} \in \left[0,\frac{1}{n!}\right).$$
3. As $f(x)=f(1-x)$ for all $x\in[0,1]$, proving $f^{(j)}(0) \in \mathbb{Z}$ should imply  $f^{(j)}(1) \in \mathbb{Z}$.
For brevity, we denote the binomial coefficients of $(1-x)^n$ by $a_{0},a_{1},\ldots,a_{n}$. So,
$$x^{n}(1-x)^{n}=x^{n}(a_{0}+a_{1}x+\dots+a_{n}x^{n}).$$
Note that the power of $x$ in each term of $f$ is at least $n$. Hence for  $j < n$, we have positive powers of $x$ in each term of  $f^{(j)}(x)$. For $j \geq n$, we have 
$$f^{(j)}(x)=\frac{(c_{j-n}+c_{j-n+1}x+\dots+c_{n}x^{2n-j})}{n!}$$
where $c_{k}=a_{k}(n+k)(n+k-1)\ldots(n+k-j+1)$ for $k=j-n,\ldots,2n$. Putting $x=0$ we get
 $$f^{(j)}(0)=\frac{c_{j-n}}{n!}=\frac{j! \textit{ } a_{j-n}}{n!}.$$
 Now, $n!$ divides $j!$ because $j \geq n$ and $a_{j-n}$ is an integer by assumption. Thus,
 $$f^{(j)}(0)\in \mathbb{Z}  \text{ for }  j \geq n.$$
\end{proof}

The nice rational approximation for a natural power of $e$ requires another function corresponding to the Niven's polynomial of degree $2n$. This is based on property (3) in the result above. Given $k\in \mathbb{N}$ and the Niven's polynomial $f_{n}$ of degree $2n$, we define a function $F_{n,k}:[0,1]\to \mathbb{R}$ as
$$F_{n,k}(x)=k^{2n}f_{n}(x)-k^{2n-1}f^{(1)}_{n}(x)+\cdots+f^{(2n)}_{n}(x)= \sum_{i=0}^{2n}(-k)^{2n-i}f_{n}^{(i)}(x).$$ From the properties of Niven's polynomial, it can be easily verified that $F_{n,k}(0)$ and $F_{n,k}(1)$ are integers.

\begin{thm} \label{thm 3.3.2}
For any $k\in \mathbb{N}$, we have $\left|F_{n,k}(1)e^{k}-F_{n,k}(0)\right| \to 0$ as $n \to \infty$. In other words $\frac{F_{n,k}(0)}{F_{n,k}(1)} \xrightarrow{*} e^{k}.$
\end{thm}
\begin{proof}
Observe that
$$F_{n,k}(x)=k^{2n}f_{n}(x)-k^{2n-1}f^{(1)}_{n}(x)+\cdots+f^{(2n)}_{n}(x).$$
Thus,
$$F'_{n,k}(x)=k^{2n}f^{(1)}_{n}(x)-k^{2n-1}f^{(2)}_{n}(x)+\cdots-kf^{(2n)}_{n}(x)=\sum_{i=1}^{2n}(-1)^{i+1}k^{2n-i+1}f_{n}^{(i)}(x).$$
Now, 
\begin{align*}\Bigl[e^{kx}F_{n,k}(x)\Bigr]' &= e^{kx}\left(kF_{n,k}(x)+F'_{n,k}(x)\right) \\
& =e^{kx}\left(\sum_{i=0}^{2n}(-1)^{i}k^{2n-i+1}f_{n}^{(i)}(x)+ \sum_{i=1}^{2n}(-1)^{i+1}k^{2n-i+1}f_{n}^{(i)}(x)\right)\\
& = e^{kx}k^{2n+1}f_{n}(x).\end{align*}
Hence,
$$\int_{0}^{1}e^{kx}k^{2n+1}f_{n}(x)dx=\Bigl[e^{kx}F_{n,k}(x)\Bigr]_{0}^{1}$$
implying that
$$F_{n,k}(1)e^{k}-F_{n,k}(0) = \int_{0}^{1}e^{kx}k^{2n+1}f_{n}(x)dx.$$

As $0 \le f_{n}(x)<\frac{1}{n!}$ and $0 < e^{kx} \le e^{k}$, for all $x \in [0,1]$, and the integrand is positive on $(0,1)$, we have
$$0<\int_{0}^{1}e^{kx}k^{2n+1}f_{n}(x)dx < \frac{e^{k}k^{2n+1}}{n!} \to 0 \ \textrm{as} \ n \to \infty.$$
(We used that $\frac{a^l}{l!} \to 0$ as $l \to \infty$ for any $a>0$.)  
Thus, $0< \left|F_{n,k}(1)e^{k}-F_{n,k}(0)\right|\to 0$ as $n \to \infty$. This gives us a nice rational approximation for $e^{k}$ and proves its irrationality. 
\end{proof}

\subsection{Rational powers}
We want to generalise further for $e^r$ where $0\neq r \in \mathbb{Q}$. Here again Niven's polynomial is helpful.  Given a non-zero rational number $r=\frac{p}{q}$ and Niven's polynomial $f_{n}$ of degree $2n$, the function that we need in this case is 
\begin{align*}
    F_{n,r}(x)  &= p^{2n}f_{n}(x)- p^{2n-1}qf^{(1)}_{n}(x)+ \cdots + q^{2n}f^{(2n)}_{n}(x) \\
    & = \sum_{i=0}^{2n}(-1)^{i}p^{2n-i}q^{i}f_{n}^{(i)}(x) .
\end{align*}
 We shall state the result without a proof because the proof is similar to that of Theorem \ref{thm 3.3.2}.

\begin{thm}\label{thm 3.4.1}
Given $e^r$ for some non-zero $\frac{p}{q}=r\in\mathbb{Q}$, we have  $\left|F_{n,r}(1)e^{r}-F_{n,r}(0)\right| \to 0$ as $n \to \infty$. In other words $\frac{F_{n,r}(0)}{F_{n,r}(1)} \xrightarrow{*} e^{r}.$
\end{thm}

From Theorem \ref{thm 3.4.1}, it is evident that $e^{r}$ is irrational for all non-zero $r \in \mathbb{Q}$. Let us see an application of this result to prove irrationality of a different number.

\begin{prop}
$ln(r)$ is irrational for all positive $ r \in \mathbb{Q} -\{1\}$.
\end{prop}
\begin{proof}
Assume that $ln(r) \in \mathbb{Q}$. Let $\alpha = ln(r)$. Then $\alpha \neq 0$ as $r \neq 1$. Thus, $e^\alpha$ is not rational. But, by definition, $e^\alpha = r \in \mathbb{Q}$. Thus contradiction. 
\end{proof}

\section{Trigonometric irrationals and related results}

\subsection{$\sin{\left(\frac{1}{m}\right)}$ and $\cos{\left(\frac{1}{m}\right)}$}
In this section, we shall find nice rational approximations for some trigonometric irrational numbers. 

\begin{thm}
\label{thm 5.1}
$\sin{\left(\frac{1}{m}\right)}$ and $\cos{\left(\frac{1}{m}\right)}$ are irrational for all $m \in \mathbb{N}$.
\end{thm}

\begin{proof}
For the proof of this theorem, we shall use the following Maclaurin series for $\sin{(x)}$ and $\cos{(x)}$.
$$
\sin{(x)}=x-\frac{x^3}{3!}+\frac{x^5}{5!}-\frac{x^7}{7!}+ \cdots
=\sum_{k=0}^{\infty} \frac{(-1)^k x^{2k+1}}{(2k+1)!}
$$
and
$$
\cos{(x)}=1-\frac{x^2}{2!}+\frac{x^4}{4!}-\frac{x^6}{6!}+ \cdots
=\sum_{k=0}^{\infty} \frac{(-1)^k x^{2k}}{(2k)!}.
$$
So,
\begin{align*}
& m^{4n-1} \ (4n-1)! \ \sin{\left(\frac{1}{m}\right)} \\ 
= & m^{4n-1} \ (4n-1)! \ \sum_{k=0}^{\infty} \frac{(-1)^k} {m^{2k+1}(2k+1)!} \\
= & \sum_{k=0}^{2n-1} \left(\frac{(4n-1)!}{(2k+1)!}\right) \ (-1)^k \ m^{\left(4n-2k-2\right)} + \sum_{k=2n}^{\infty} \frac{(-1)^k \ m^{4n-1} \ (4n-1)!} {m^{2k+1}(2k+1)!}
\end{align*}
Taking $$p_n = \sum_{k=0}^{2n-1} \left(\frac{(4n-1)!}{(2k+1)!}\right) \ (-1)^k \  m^{\left(4n-2k-2\right)}$$ and $q_n=m^{4n-1} \ (4n-1)!$, we see that
\begin{align*}
\left|q_n \sin{\left(\frac{1}{m}\right)} - p_n\right| & = \left|\sum_{k=2n}^{\infty} \frac{(-1)^k \ m^{4n-1} \ (4n-1)!} {m^{2k+1}(2k+1)!}\right|
\\
& 
\le \sum_{k=2n}^{\infty} \frac{m^{4n-1} \ (4n-1)!}{m^{2k+1}(2k+1)!} \\
& =\frac{1}{m^2} \frac{1}{4n(4n+1)}+\frac{1}{m^4}\frac{1}{4n(4n+1)(4n+2)(4n+3)}+\cdots \\
& < \frac{1}{m^2} \frac{1}{(4n)^2}+\frac{1}{m^4}\frac{1}{(4n)^4}+....=\frac{1}{m^2 (4n)^2 -1}.
\end{align*}
As $n \to \infty $, we have $\frac{1}{m^2 (4n)^2 -1} \to 0 $.
Also, 
$$
q_n \sin{\left(\frac{1}{m}\right)} - p_n = \sum_{j=n}^{\infty} \left[m^{4n-1} (4n-1)! \left\{ \frac{1}{m^{4j+1} (4j+1)!} - \frac{1}{m^{4j+3} (4j+3)!}\right\}\right]
$$
is positive as each term of the summation is positive. So,
$$
0 \neq \left|q_n \sin{\left(\frac{1}{m}\right)} - p_n\right| \to 0
$$
as $ n \to \infty $.
So, $\left\{\frac{p_n}{q_n}\right\}_{n \in \mathbb{N}}$ is a nice rational approximation of $\sin{\left(\frac{1}{m}\right)}$. So, $\sin{\left(\frac{1}{m}\right)}$ is irrational. The proof of irrationality of $\cos{\left(\frac{1}{m}\right)}$ is analogous.
\end{proof}

An argument similar to that used in Theorem \ref{thm 5.1} can be used to prove that $\cos{\left(\frac{1}{\sqrt{m}}\right)}$ is irrational for all $m \in \mathbb{Z}$.

\subsection{A general result and irrationality of $\pi$ }

The main result of this section is as follows. 

\begin{thm}
\label{thm 5.3}
If $\alpha \in (0, \pi ]$ is a rational number, then at least one of $\sin{\alpha}$ and $\cos{\alpha}$ is irrational.
\end{thm}

The proof of this theorem requires an auxilliary result.

\begin{lem}
\label{lem 5.2}
If $0 \neq \left|a_n \alpha + b_n \beta \right| \to 0$ as $ n \to \infty $, with $a_n,b_n \in \mathbb{Z}$ for all $n \in \mathbb{N} $, then at least one of $\alpha$ and $\beta$ is irrational.
\end{lem}

\begin{proof}
We shall prove it by contradiction. Suppose, both $\alpha$ and $\beta$ are rational. Let $\alpha = \frac{p}{q}$ and $\beta = \frac{r}{s}$, where $p,r \in \mathbb{Z}$ and $q,s \in \mathbb{Z}-\{0\}$.
\\
So, 
$$
\left| a_n \alpha + b_n \beta \right|=\left|\frac{(a_n p s + b_n q r)}{qs} \right|
$$

Now, $(a_n p s + b_n q r) \in \mathbb{Z}$, and as $\left| a_n \alpha + b_n \beta \right| \neq 0 $,

$$
\left|\frac{(a_n p s + b_n q r)}{qs} \right| \ge \frac{1}{qs}
$$
So, $\left| a_n \alpha + b_n \beta \right| \ge \frac{1}{qs}$ cannot tend to $0$ as $n \to \infty $.
So, our supposition must be false, i.e. at least one of $\alpha$ and $\beta$ must be irrational.

\end{proof}

{\em Proof of the Theorem}.

Let $\alpha = \frac{p}{q}$ for some $p,q \in \mathbb{Z}^{+}$. Let us define a set S as 
$$
S = \left\{ x+iy :x,y \in \mathbb{Z} \right\}.
$$
Define a function $F_n:[0,1] \to \mathbb{C} $ by
$$
F_n= (ip)^{2n} f_n - (ip)^{2n-1} q f'_n+(ip)^{2n-2} q^2 f''_n - ... + q^{2n} f_n^{(2n)},
$$
where $$f_n(x) =\frac{ x^n (1-x)^n}{n!}$$ for all $x \in [0,1].$
Clearly, $F_n(0),F_n(1) \in S$ and $0 \le f_n(x) \le \frac{1}{n!}$ for all $x \in [0,1]$.
Now,
\begin{align*}
F'_n & =(ip)^{2n} f'_n - (ip)^{2n-1} q f''_n+(ip)^{2n-2} q^2 f'''_n - ...-(ip) q^{2n-1} f_n^{(2n)} + q^{2n} f_n^{(2n+1)} \\
& =(ip)^{2n} f'_n - (ip)^{2n-1} q f''_n+(ip)^{2n-2} q^2 f'''_n - ... -(ip)q^{2n-1} f^{(2n)}_n
\end{align*}
as $f_n$ is a polynomial of degree $2n$.
So,
$$
F'_n +\frac{ip}{q} F_n = \frac{(ip)^{2n+1}}{q} f_n
$$
This is a linear differential equation for $F_n$, which gives the solution
$$
F_n(x) e^{\frac{ipx}{q}} = \int \frac{(ip)^{2n+1}}{q} e^{\frac{ipx}{q}} f_n(x) dx + C, \textrm{ for some constant } C \ .
$$
Thus,
$$
\left[ F_n(x) e^{\frac{ipx}{q}} \right]_{0}^{1} = \int_{0}^{1} \frac{(ip)^{2n+1}}{q} e^{\frac{ipx}{q}} f_n(x) dx 
$$
implying that

\begin{equation*}
\
F_n(1) e^{\frac{ip}{q}}-F_n(0)=\frac{(ip)^{2n+1}}{q} \int_{0}^{1} e^{\frac{ipx}{q}} f_n(x) dx
\end{equation*}

It should be clear that the RHS of the above expression is non-zero. To see why, note that 
$ e^{\frac{ipx}{q}}=\cos{\left(\frac{px}{q}\right)}+i \sin{\left(\frac{px}{q}\right)} $, so, the imaginary part of the integral is 

$$
\int_{0}^{1} \sin{\left(\frac{px}{q}\right)} f_n(x) dx,
$$
which must be non-zero as the integrand is positive on $(0,1)$ for $0 < \frac{p}{q} \le \pi$. So, the R.H.S.  has non-zero real part. Now,
\begin{align*}
0 & \neq \left| \operatorname{Re}{\left( \frac{(ip)^{2n+1}}{q} \int_{0}^{1} e^{\frac{ipx}{q}} f_n(x) dx \right)}\right| \\
& =\left| \frac{p^{2n+1}}{q} \int_{0}^{1} \sin{\left(\frac{px}{q}\right)} f_n(x) dx\right|\\
& \le \frac{p^{2n+1}}{q} \int_{0}^{1} \frac{1}{n!} dx\\
& = \frac{p^{2n+1}}{n! \ q} \to 0 \textrm{ as } n \to \infty \ .
\end{align*}

So, for the LHS, we have 

$$
0 \neq \left| \operatorname{Re}{\left(F_n(1) e^{\frac{ip}{q}}-F_n(0)\right)}\right| \to 0 \textrm{ as } n \to \infty \ .
$$

Let $F_n(0)=a_n+b_n i$ and $F_n(1)=c_n+d_n i$ . As $F_n(0),F_n(1) \in S$, we have  $a_n,b_n,c_n,d_n \in \mathbb{Z}$ . Now, 
\begin{align*}
\operatorname{Re}{\left(F_n(1) e^{\frac{ip}{q}}-F_n(0)\right)} & = \operatorname{Re}{\left(\left(c_n+d_n i\right) \left( \cos{\left( \frac{p}{q}\right)} + i \sin{\left( \frac{p}{q}\right)}  \right) - \left( a_n + b_n i \right) \right)} \\
& = c_n \cos{\left( \frac{p}{q}\right)} - d_n \sin{\left( \frac{p}{q}\right)} -a_n.
\end{align*}
Therefore,
$$
 \ 0 \neq \left| c_n \cos{\left( \frac{p}{q}\right)} - d_n \sin{\left( \frac{p}{q}\right)} -a_n\right| \to 0 \ \textrm{as} \ n \to \infty \ .
$$
 So, from Lemma \ref{lem 5.2}, at least one of $\sin{\left( \frac{p}{q}\right)}$ and $ \cos{\left( \frac{p}{q}\right)}$ must be irrational.

\qed

The above theorem can have wide applications. The most important application is probably that we can conclude about irrationality of $\pi$ from the above theorem. It is presented in the following corollary.

\begin{cor}
$\pi$ is irrational.
\end{cor}

\begin{proof}
Suppose, $\pi$ is rational. As $\pi \in (0,\pi]$, then from Theorem \ref{thm 5.3}, at least one of $\sin{(\pi)}$ and $\cos{(\pi)}$ must be irrational. But $\sin{(\pi)}=0$ and $\cos{(\pi)}=-1$ are both rational. So, our supposition is false and $\pi$ is irrational.
\end{proof}

One more application of Theorem \ref{thm 5.3} is illustrated in the following example.

\begin{exmp}
We want to know about irrationality of $\sin^{-1}{\left(\frac{1}{\sqrt{26}}\right)}$ .
\\
We take $\alpha=\sin^{-1}{\left(\frac{1}{\sqrt{26}}\right)}$. So, $\cos{\left(\alpha\right)}=\sqrt{1-\left(\frac{1}{\sqrt{26}}\right)^2} = \frac{5}{\sqrt{26}}$ . Now, both $\sin{\left(2\alpha\right)}= 2 \sin{\left( \alpha\right)}\cos{(\alpha)} = \frac{5}{13}$ and $\cos{\left(2\alpha\right)}=\cos^2{(\alpha)}-\sin^2{(\alpha)}=\frac{12}{13}$ are rational.
\\
As $0<\alpha < \frac{\pi}{2}$, $0 < 2\alpha < \pi$ , if $2\alpha$ is rational, then both $\sin{\left(2\alpha\right)} $ and $\cos{\left(2\alpha\right)}$ cannot be rational from Theorem \ref{thm 5.3} . So, $2\alpha$ must be irrational and from Lemma \ref{lem 3.5}, $\alpha = \sin^{-1}{\left(\frac{1}{\sqrt{26}}\right)}$ is irrational.

\end{exmp}

\newpage 
\section*{Conclusion}

We would like to sum up what we have done. We established that finding a nice rational approximation for a number is sufficient to prove its irrationality. Then we derived some related results. We used some specific methods to find nice rational approximations for some irrational numbers, such as $\sqrt{m} \notin \mathbb{Z} \ [ m \in \mathbb{N}]$, $e$, $e^2$, $\sin{\left(\frac{1}{n}\right)}$ etc. In some cases, where we did not find a nice rational approximation, we used our derived results to prove irrationality of some numbers, such as algebraic irrationals, indirectly. Finally, we proved a theorem, from which, we can prove irrationality of $\pi$.

At this point, we would like to mention that many other methods apart from the method of finding a nice rational approximation can also be used to prove irrationality of numbers and some specific methods may be found to be more useful in some specific cases. There are a large number of results related to irrational numbers. Before concluding our discussion, we mention some of them, that are related to the results which we have obtained so far.
\begin{enumerate}
    \item If $\alpha$ is a non-zero algebraic number, then $e^{\alpha}$ is irrational.
    \item The trigonometric functions $\sin{\alpha}, \ \cos{\alpha}, \ \tan{\alpha}$ are irrational for non-zero rational $\alpha$ .
    \item Any non-zero value of an inverse trigonometric function is irrational for a rational value of the argument.
\end{enumerate}


\begin{thebibliography}{}



\bibitem{2} Hardy, G. H. and Wright, E. M., An introduction to the theory of numbers. (First ed.), Oxford: Clarendon Press, 1938.

\bibitem{3}  Ireland K. and  Rosen M., A classical introduction to modern number theory, 2nd ed., Springer-Verlag, New York, 1990.

\bibitem{1} Niven, I., Irrational Numbers, The Carus Mathematical Monographs, The Mathematical Association of America, 1958.

\bibitem{Roth} Roth, K. F.,  Rational approximations to algebraic numbers. Mathematika, 2(1), 1-20, 1955. 

\bibitem{4} Spivak, M., Calculus, 2nd ed., Publish or Perish, Wilmington, Del., 1967.


\end{thebibliography}
\end{document}